\documentclass[oneside, 11pt]{amsart} \topmargin      -10mm
\usepackage{amsmath, amssymb, amsfonts, amstext, amsthm, amscd}
\usepackage[mathscr]{euscript}
\usepackage{enumerate}
\usepackage{tikz,color,soul}
\usepackage[english]{babel}
\usepackage{chngcntr}
\usetikzlibrary{arrows.meta, positioning,arrows}
\newtheorem{theorem}{Theorem}[section]
\newtheorem{proposition}[theorem]{Proposition}
\newtheorem{lemma}[theorem]{Lemma}
\newtheorem{corollary}[theorem]{Corollary}
\theoremstyle{definition}
\newtheorem{definition}[theorem]{Definition}
\newtheorem{eg}[theorem]{Example}
\newtheorem{remark}[theorem]{Remark}
\newtheorem*{note}{Note}
\numberwithin{equation}{section}
\title{On the topology of real Bott manifolds}
\author{Raisa Dsouza}
\address{Department of Mathematics, IIT Madras, Chennai 600036, India} 
\email{raisadsouza1989@gmail.com} 
\date{}

\begin{document}

\begin{abstract}
 The main aim of this article is to give a necessary and sufficient condition for a real Bott manifold to admit a spin structure and further give a combinatorial characterization for the spin structure in terms of the associated acyclic digraph. \\
 \vspace{5mm}\noindent
\end{abstract}

\footnotetext{{\bf 2010 Subject Classification:} Primary: 55R99, Secondary: 57S25}

\maketitle

\section{Introduction}
A Bott tower is an iterated sequence of fibre bundles with fibre at each stage being $\mathbb C \mathbb P^1$. The manifold at each stage of the sequence is called a Bott manifold. In particular these manifolds are smooth projective complex toric varieties. They were constructed in \cite{grossbergkarshon} by Grossberg and Karshon who show that a Bott-Samelson variety can be deformed to a Bott manifold. Apart from \cite{grossbergkarshon}, the topology and geometry of these objects have been studied by Civan and Ray in \cite{civanyusuf} and \cite{civanray}. In particular, in \cite{civanyusuf}, Civan looks at these manifolds as a special class of smooth toric varieties and studies their construction in terms of {\it Bott numbers}. 

Indeed, from the viewpoint of toric topology, Bott manifolds can be seen to also have the structure of a quasitoric manifold \cite{davisjanus} with the quotient polytope being the $n$-dimensional cube $I^n$ where $n$ is the complex dimension of the Bott manifold. 

Recently there has been extensive work on the topology and geometry of Bott manifolds viewed as a quasitoric manifold (see \cite{masudapanovsemifree} and \cite{masudachoigeneralized}). These works are especially related to the problem of cohomological rigidity of Bott manifolds or more generally of quasitoric manifolds.

There has also been a parallel study on the topology of real Bott manifolds. These manifolds are constructed as iterated $\mathbb R \mathbb P^1=S^1$-bundles and can be viewed as a special example of a small cover defined by Davis and Januszkiewicz in \cite{davisjanus} (see for example \cite{masudakamishimacohorigid}, \cite{masudachoidigraph} ). Moreover, the data of the characteristic function for this small cover is encoded by an upper triangular nilpotent matrix $C\in M_n(\mathbb{Z}_2)$. We call $C$ the {\it Bott matrix} and its entries $c_{i,j}$ $1\leq i<j\leq n$, \emph{Bott numbers.}

It is known \cite{masudakamishimacohorigid} that a real Bott manifold is orientable if and only if 
\begin{equation}\label{orifirst}
 \sum_{j=1}^nc_{i,j}\equiv0\bmod2\qquad\text{for}\qquad 1\leq i\leq n.
\end{equation}
Indeed (\ref{orifirst}) is equivalent to the vanishing of the first Stiefel-Whitney class of the real Bott manifold in the cohomology ring with $\mathbb Z_2$-coefficients. Our motivation in this article is to give necessary and sufficient conditions for the vanishing of the higher Stiefel-Whitney classes in terms of algebraic identities in the $c_{i,j}$'s. This is of interest since the vanishing of the Stiefel-Whitney classes have topological interpretations. In particular, in Theorem \ref{spinresult} we derive a closed formula for the second Stiefel-Whitney class in the $c_{i,j}$'s and hence obtain a necessary and sufficient condition for a real Bott manifold to be spin. 

In Section \ref{genbottmatsection}, we consider a more general real Bott manifold $M(B)$, where $B=(b_{i,j})\in \mathcal B(n)$ \cite[Section 1, page 2]{masudachoidigraph}, and give a necessary and sufficient condition in Theorem \ref{genspinresult} for $M(B)$ to admit a spin structure in terms of algebraic identities on the $b_{i,j}$'s. We wish to mention here that the spin structure of real Bott manifolds has recently been studied by G{\c a}sior in \cite{gasiorspin}, where again a necessary and sufficient condition has been given for $M(B)$ to be spin. We would like to point out that our characterization in Theorem \ref{genspinresult} is more intrinsic and the main result \cite[Theorem 1.2]{gasiorspin} follows as a corollary (see Corollary \ref{gencrspinsub}). 

In \cite{masudachoidigraph}, Choi, Masuda and Oum associate an acyclic directed graph to a real Bott manifold. Indeed this is a directed graph whose adjacency matrix is the Bott matrix. Apart from other results, in \cite[Lemma 4.1]{masudachoidigraph}, they give a criterion for orientability and symplectic structure on the Bott manifold in terms of the combinatorics of the associated digraph. Motivated by their work we obtain a digraph-characterization for the spin structure on these manifolds in Theorem \ref{digraphspin}. 

In Theorem \ref{n-1th sfclass} we give a formula for the $(n-1)^\text{th}$ Stiefel-Whitney class. Further, in Theorem \ref{vanishswnosBT} we prove that real Bott manifolds are null-cobordant by showing that all their Stiefel-Whitney numbers vanish. Moreover, in Corollary \ref{orcoborBT}  we prove  that an oriented real Bott manifold is orientedly null-cobordant.

\subsection{Notations and Conventions}

In this section we recall the definition of a Bott tower and fix some notations (see \cite{grossbergkarshon}).

A \emph{Bott tower} is a sequence of smooth complete complex toric varieties which are constructed iteratively as follows:

Let $Y_1=\mathbb C\mathbb P^1$. Let $L_2$ be a holomorphic complex line bundle on $\mathbb C\mathbb P^1$.  We then let $Y_2=\mathbb P(\mathbf{1}\oplus L_2)$ where $\mathbf{1}$ is the trivial line bundle on $\mathbb C\mathbb P^1$. Then $Y_2$ is a $\mathbb C\mathbb P^1$ bundle over $\mathbb C\mathbb P^1$ which is a Hirzebruch surface. We can iterate this process for $2\leq j\leq n$, where at each step, $L_j$ is a complex line bundle over $Y_{j-1}$, and the variety $Y_j=\mathbb P(\mathbf{1}\oplus L_j)$ is a $\mathbb C\mathbb P^1$ bundle over $Y_{j-1}$. The sequence $$Y_n\rightarrow Y_{n-1}\rightarrow\cdots\rightarrow Y_1\rightarrow\{\text{pt}\}$$ is called an $n$-step Bott tower. Each $Y_j$ is called a Bott manifold.

\begin{definition}\label{Bottfan}
In fact an $n$-dimensional Bott manifold is a smooth complete toric variety of dimension $n$ whose fan $\Delta$ can be described as follows:

We take a collection of integers $\{a_{i,j}\},\ 1\leq i<j\leq n$. Let $e_1,\cdots,e_n$ be the standard basis vectors of $\mathbb R^n$. Let $v_j=e_j$ for $1\leq j\leq n$, \[v_{n+j}=-e_j+\sum\limits_{k=j+1}^na_{j,k}\, e_k\] for $1\leq j\leq n-1$ and $v_{2n}=-e_n$. We define the fan $\Delta$ in $\mathbb R^n$ consisting of cones generated by the set of vectors in any sub collection of $\{v_1,v_2,\cdots,v_n,v_{n+1},\cdots,v_{2n}\}$ which does not contain both $v_i$ and $v_{n+i}$ for $1\leq i\leq n$.
\end{definition}

\begin{definition}\label{Bottqt}
 We can also view a Bott manifold as a quasi-toric manifold (see \cite{davisjanus}) over the $n$-cube $I^n$ which is a simple convex polytope of dimension $n$. If we index the $2n$ facets of $I^n$ by $F_1,F_2,\ldots,F_n,F_{n+1},\ldots,F_{2n}$, then the characteristic function, $\lambda$ is defined on the collection of facets, $\mathcal F$ to $\mathbb Z^n$ as follows: $\lambda(F_j)=e_j$ for $1\leq j\leq n$, \[\lambda(F_{n+j})=-e_j+\sum_{k=j+1}^n a_{j,k}\cdot e_{j+k}\] for $1\leq j\leq n-1$ and $\lambda(F_{2n})=-e_n$. 
\end{definition}

\subsubsection{Real Bott manifold}
We shall call the real part of the $n$-dimensional complex Bott manifold as the real $n$-dimensional Bott manifold.

In particular, $(Y_2)_{\mathbb R}$ is an $\mathbb R\mathbb P^1$ bundle over $(Y_1)_{\mathbb R}=\mathbb R\mathbb P^1$. Iteratively we construct $(Y_j)_{\mathbb R}$ as an $\mathbb R\mathbb P^1$ bundle over $(Y_{j-1})_{\mathbb R}$ for $2\leq j\leq n$.  The real $n$-dimensional Bott manifold $(Y_n)_{\mathbb R}$ is indeed the real toric variety associated to the fan $\Delta$ described above (see \cite[Section 2.4]{jurkiewiczthesis} and \cite{umafundamental}).

\begin{definition}\label{rbsc}
As in the complex case we can also view $(Y_n)_{\mathbb R}$ as a {\it small cover} over the simple convex polytope $I^n$, where the characteristic map $\lambda$ is defined on the collection of facets, $\mathcal F$ to ${\mathbb Z}_2^n$ as follows: $\lambda(F_j)=e_j$ for $1\leq j\leq n$,  \[\lambda(F_{n+j})=e_j+\sum_{k=j+1}^n c_{j,k}\cdot e_k\] for $1\leq j\leq n-1$ and $\lambda(F_{2n})=e_n$ where $c_{i,j}=a_{i,j}\bmod2$ for $1\leq i<j\leq n$. Thus $(Y_n)_{\mathbb R}$ is homeomorphic to the identification space ${\mathbb Z_2^n\times I^n}/\sim$ where $(t,p)\sim(t^{\prime},p^{\prime})$ if and only if $p=p^{\prime}$ and $t\cdot (t^{\prime})^{-1}\in G_{F(p)}$. Here $F(p)=F_1\cap\cdots\cap F_l$ is the unique face of $I^n$ which contains $p$ in its relative interior and $G_{F(p)}$ is the rank-$l$ subgroup of $\mathbb Z_2^n$ determined by the span of $\lambda(F_1),\ldots,\lambda(F_l)$.
\end{definition}

The topological structure of an $n$-dimensional real Bott manifold is completely determined by the simple convex polytope $I^n$ and the data encoded by the matrix \begin{equation} \label{Bottmatrix} C=(c_{i,j})\in M_n(\mathbb Z_2)\end{equation} where $c_{i,j}=0$ for $i\geq j$. Note that the $i$th row of $C+I$ is $\lambda(F_{n+i})\in \mathbb Z_2^n$ for $1\leq i\leq n$. We call $C$ the {\it Bott matrix}. Thus $(Y_n)_{\mathbb R}$ is the real Bott manifold associated to $C$. 

The $2$-dimensional real Bott manifold is the torus or the Klein bottle depending on whether $c_{1,2}=0$ or $c_{1,2}=1$. The $3$-dimensional real Bott manifold is an  $\mathbb R\mathbb P^1$ bundle over the torus or the Klein bottle whose topological structure depends on $c_{1,2},c_{1,3}$ and $c_{2,3}$.

\begin{note}
 In this article, since we are mainly interested in the study of real Bott manifolds, for notational simplicity we shall henceforth denote $(Y_n)_{\mathbb R}$ by $Y_n$. If we wish to specify the associated Bott matrix we shall denote $Y_n$ by $Y_n(C)$.
\end{note}

\section{Stiefel-Whitney classes of $Y_n$}
In this section we briefly recall the description of the cohomology ring with $\mathbb Z_2$-coefficients of $Y_n$. A discription of $H^*(Y_n;\mathbb Z_2)$ has been given earlier in \cite[Section 2]{masudakamishimacohorigid}. We also recall the formula for total Stiefel-Whitney class as given in \cite[Section 6]{davisjanus}. In Theorem \ref{simrf} we give a recursive formula for the total Steifel-Whitney class as well as the $k$th Steifel-Whitney class of $Y_n$ in terms of those of $Y_{n-1}$.  

\begin{proposition}\label{cohoring}
 Let $\mathcal R:=\mathbb Z_2[x_1,x_2,\cdots,x_{2n}]$ and let $\mathcal I$ denote the ideal in $\mathcal R$ generated by the following set of elements
 \begin{equation}\label{ideal}
  \left\{x_j\,x_{n+j}\,,\,x_1+x_{n+1}\,,\,x_j+x_{n+j}+\sum_{i=1}^{j-1}c_{i,j}\,x_{n+i}\,\forall\ 2\leq j\leq n\right\}
 \end{equation}
 As a graded $\mathbb Z_2$-algebra, $H^*(Y_n;\mathbb Z_2)$ is isomorphic to $\mathcal{R/I}$.
\end{proposition}

Let $w_k(Y_n)$ denote the $k^{\text{th}}$ Stiefel-Whitney class of $Y_n$ for $0\leq k\leq n$ with the understanding that $w_0(Y_n)=1$. Then $w(Y_n)=1+w_1(Y_n)+\cdots+w_n(Y_n)$ is the total Stiefel-Whitney class of $Y_n$.

\begin{proposition}\label{rf} (i) Under the isomorphism of $H^*(Y_n;\mathbb Z_2)$ with $\mathcal{R/I}$ we have the identification
\begin{equation}\label{sw} w(Y_n)=\prod_{i=1}^{2n} (1+ x_i)\end{equation}  where $x_i$ for $1\leq i\leq 2n$ satisfy (\ref{ideal}).
  
 (ii) We further have the following recursive formula
 \begin{equation}\label{rf1}w(Y_n)=w(Y_{n-1})\cdot (1+x_n)(1+x_{2n}),\end{equation} where \begin{equation}\label{nrels} x_n\cdot x_{2n}=0, x_n=x_{2n}-\sum_{i=1}^{n-1} c_{i,n}x_{n+i}.\end{equation}
\end{proposition}

\begin{proof}
 The proof of (i) follows readily by applying \cite[Corollary 6.8]{davisjanus} for $Y_n$.

 Now, we shall prove (ii). 

 Note that the defining Bott matrix for $Y_{n-1}$ is the $(n-1)\times(n-1)$ submatrix of $C$ obtained by deleting the $n$th row and the $n$th column of $C$.
 Moreover, let $\pi_n:Y_n\rightarrow Y_{n-1}$ denote the canonical projection of the $\mathbb{RP}^1$-bundle. Then via pullback along $\pi_n^*$, $H^*(Y_{n-1};\mathbb Z_2)$ can be identified with the subring $\mathcal{R'/I'}$ of $\mathcal{R/I}$ where $\mathcal{R'}=\mathbb Z_2[x_1,x_2\ldots,x_{n-1},x_{n+1},\ldots, x_{2n-1}]$ and $\mathcal I'$ is the ideal generated by the relations 
 \begin{equation}\label{st1}\{x_ix_{n+i}~,~ x_i-x_{n+i}+\sum_{j=1}^{i-1}c_{j,i}\cdot x_{n+j}~~\mbox{for}~~ 1\leq i\leq n-1\}.\end{equation}
                                                                                                                              
 Since $Y_n$ is an $\mathbb {RP}^1$-bundle over $Y_{n-1}$, we further have the following presentation of $H^*(Y_n;\mathbb Z_2)$ as an algebra over $H^*(Y_{n-1};\mathbb Z_2)$: 
 \begin{equation}\label{presind}H^*(Y_n;\mathbb Z_2)\simeq H^*(Y_{n-1};\mathbb Z_2)[x_n,x_{2n}]/ J\end{equation}
 where $J$ is the ideal generated by the relations \begin{equation}\label{relnind} x_n\cdot x_{2n},~~x_{n}-x_{2n}+\sum_{i=1}^{n-1}c_{i,n}x_{n+i}.\end{equation}
 Furthermore, via $\pi_n^*$ we can identify $w(Y_{n-1})$ with the expression\begin{equation}\label{sw1} w(Y_{n-1})=\prod_{i=1}^{n-1} (1+x_i)\cdot\prod_{i=n+1}^{2n-1} (1+ x_i)\end{equation} in $\mathcal R'$ where $x_i$ for $1\leq i\leq n-1$ and $n+1\leq i\leq 2n-1$ satisfy the relations (\ref{st1}). Now by (\ref{sw}) and (\ref{sw1}), (ii) follows.
\end{proof}

\begin{theorem}\label{simrf} 
 \begin{enumerate}\item[(i)] The following hold in the $\mathbb Z_2$-algebra $\mathcal{R/I}$: \begin{equation}\label{rf1'}w(Y_n)=w(Y_{n-1})(1+\sum_{i=1}^{n-1}c_{i,n}x_{n+i}),\end{equation}
 \begin{equation}\label{rf2} w_{k}(Y_n)=w_{k}(Y_{n-1})+w_{k-1}(Y_{n-1})\cdot (\sum_{i=1}^{n-1}c_{i,n}x_{n+i}) \end{equation}
 for $n\geq 2$ and $1\leq k\leq n$.
 \item[(ii)] For every $1\leq k\leq n$, $w_{k}(Y_n)$ is a $\mathbb Z_2$-linear combination of square free monomials of degree $k$ in the variables $x_{n+1},\ldots,x_{2n-1}$ modulo $\mathcal I$.
 \end{enumerate} 
\end{theorem}

\begin{proof}
 The equation (\ref{rf1}) reduces to (\ref{rf1'}) by applying (\ref{nrels}). Note that under the isomorphism of the graded algebras $H^*(Y_n;\mathbb Z_2)$ and $\mathcal{R/I}$, $w_k(Y_n)\in H^k(Y_n;\mathbb Z_2)$ corresponds to a polynomial of degree $k$ in $x_i ,1\leq i\leq 2n$ modulo $\mathcal I$ for $1\leq k\leq n$.  Thus we get (\ref{rf2}) by comparing the degree $k$-terms on either side of (\ref{rf1'})and (i) follows. 
 
 Observe that by applying (\ref{st1}), in $\mathcal{R'/I'}$ and hence in $\mathcal{R/I}$, we can substitute for $x_i$ in terms of $x_{n+1},\ldots, x_{n+i}$ modulo $\mathcal I$ using the equality \begin{equation}\label{redvar} x_{i}=x_{n+i}+\sum_{j=1}^{i-1}c_{j,i}\cdot x_{n+j}.\end{equation}
 In particular, $w_{k}(Y_{n-1})$ (resp. $w_{k-1}(Y_{n-1})$) can be written as a polynomial of degree $k$ (resp. $k-1$) in $x_{n+i}, ~~1\leq i\leq n-1$.  Furthermore, multiplying either side of (\ref{redvar}) with $x_{n+i}$, along with the equality $x_i\cdot x_{n+i}=0$ gives 
 \begin{equation}\label{sqfree} x_{n+i}^2=\sum_{j=1}^{i-1}c_{j,i}\cdot x_{n+i}\cdot x_{n+j}\end{equation} 
 for $1\leq i\leq n-1$.  It follows that $w_{k}(Y_{n-1})$ (resp. $w_{k-1}(Y_{n-1})$) can be expressed as $\mathbb Z_2$-linear combinations of square free monomials of degree $k$ (resp. $k-1$) in $x_{n+i}, ~~1\leq i\leq n-1$ in the algebra $\mathcal{R/I}$.  Now, assertion (ii) follows readily by applying (\ref{sqfree}) again in (\ref{rf2}).
 \end{proof}

\begin{corollary}\label{wkrecursive}
 The $k^\text{th}$ Stiefel-Whitney class of $Y_n$ can be written in terms of the $(k-1)^\text{th}$ Stiefel-Whitney classes of $Y_{n},\cdots,Y_{1}$ as follows:
 \begin{equation}
  w_k(Y_n)=\sum_{t=1}^{n-1}w_{k-1}(Y_t)\,A_{t+1}
 \end{equation}
 where $A_t=\sum_{i=1}^{t-1}c_{i,t}\,x_{n+i}$.
\end{corollary}

\begin{proof}
 The proof follows from (\ref{rf2}) by induction on $n$.
\end{proof}

The next proposition has been proved in \cite[Lemma 2.2]{masudakamishimacohorigid} for real Bott manifolds. We state and prove it here for completeness. Also see \cite{naknishiorientability} for orientability criterion for any small cover.

\begin{proposition}\label{maskamorientable}
  The real Bott manifold $Y_n$ is orientable if and only if the sum of entries in each row of the Bott matrix $C=(c_{i,j})$ are zero in $\mathbb Z_2$, that is, \begin{equation}\label{mkorientable}
      \sum_{j=1}^n c_{i,j}\equiv0\bmod2 \text{ for every } 1\leq i\leq n.
      \end{equation}
\end{proposition}

\begin{proof}
 By putting $k=1$ in (\ref{rf2}) and by induction on $n$ we get,
  \begin{equation}\label{1sw}
  w_1(Y_n)=\sum_{i=1}^{n-1}(\sum_{j=i+1}^n c_{i,j})\cdot x_{n+i}=\sum_{i=1}^{n-1}(\sum_{j=1}^n c_{i,j})\cdot x_{n+i}
 \end{equation}
 where the second equality follows from the fact that $c_{i,j}=0$ for $i\geq j$.
 
 The proposition then follows from the fact that a compact connected differentiable manifold $M$, is orientable if and only if $w_1(M)=0$ and that as a $\mathbb Z_2$-vector space, $H^1(Y_n;\mathbb Z_2)$ is isomorphic to the subspace of $\mathcal{R/I}$ freely generated by $x_{n+i},\ 1\leq i\leq n$.
\end{proof}

\section{Spin structure on real Bott manifolds}
In this section we give a necessary and sufficient condition in terms of the Bott numbers for a Bott manifold to admit a spin structure (Theorem \ref{spinresult}). 

\begin{definition}\label{spindef}
 The \emph{spinor group} $\text{Spin}(n)$ (for $n\geq3$) is the connected double cover of the special orthogonal group $SO(n)$. There exists a short exact sequence of Lie groups $$1\rightarrow\mathbb Z_2\rightarrow \text{Spin}(n)\xrightarrow{\lambda} SO(n)\rightarrow 1.$$
\end{definition}

An oriented Riemannian manifold $X$ is said to admit a spin structure if the oriented frame bundle $F$ associated to its tangent bundle, which is a principal $SO(n)$-bundle, lifts to a principal $\text{Spin}(n)$-bundle. More precisely, if there is a principal $\text{Spin}(n)$-bundle $P$ on $X$ that is a double cover of $F$.

It is further known that an $SO(n)$-bundle admits a spin structure if and only if its second Stiefel-Whitney class is zero (\cite[Theorem 1.7, pg 86]{lawsonspingeometry}). Using this criterion, we give a necessary and sufficient condition, in terms of algebraic identities in the Bott numbers, for an $n$-dimensional orientable Bott manifold to admit a spin structure. 

\begin{theorem}\label{spinresult}
 The Bott manifold $Y_n$ admits a spin structure if and only if the following two conditions are satisfied :
 \begin{enumerate}
 \item The row sums of the Bott matrix $C=(c_{i,j})$ are even. That is, for every $1\leq i\leq n$,
 \begin{equation} \label{orientable}
 \sum_{j=1}^n c_{i,j}\equiv0\bmod{2}
 \end{equation}
 \item For every $1\leq j<k\leq n$,
 \begin{equation} \label{spin}
 \underbrace{\sum_{r=1}^n c_{j,r}\,c_{k,r}}_{P_{jk}}+\underbrace{c_{j,k}\cdot\sum_{\substack{r,s=1\\ r<s}}^n c_{k,r}\, c_{k,s}}_{Q_{jk}}\equiv 0\bmod{2} 
\end{equation}
\end{enumerate}
\end{theorem}

\begin{proof}
  Condition (1), which is precisely (\ref{mkorientable}), says that $Y_n$ is orientable. We know that an orientable manifold admits a spin structure if and only if its second Stiefel-Whitney class vanishes (\cite[Theorem 1.7, pg 86]{lawsonspingeometry}). We will now prove that this is equivalent to condition (2).
  
  By Proposition \ref{cohoring} and by equation (\ref{sw}), $w(Y_n)$ can be identified with the class in $\mathcal{R/I}$ of the following term \begin{equation}\label{tsw} \prod_{j=1}^n \left(1+ x_j+ x_{n+j}+ x_j\cdot x_{n+j}\right) .\end{equation}
  Further, using the relations (\ref{ideal}) in $\mathcal I$ we can rewrite (\ref{tsw}) as 
  \begin{equation}\label{split} \prod_{j=2}^n \left(1+\sum_{i=1}^{j-1} c_{i,j}\cdot x_{n+i}\right) \end{equation}
  Furthermore, Proposition \ref{cohoring} gives an isomorphism of graded $\mathbb Z_2$-algebras, the degree $2$ term of $w(Y_n)$, namely  $w_2(Y_n)$ can be identified with the degree $2$ term of expression (\ref{split}) which is the class of the following term in $ \mathcal{R/I}$:
  \begin{equation}\label{fsw2}
  \sum_{1\leq j<k\leq n-1} \left(\sum_{r=j+1}^n\sum_{\substack{s= k+1 \\ s\neq r}}^ n c_{j,r} c_{k,s} ‎\right) x_{n+j}\,x_{n+k}‎+\sum_{k=1}^{n-2}\left(\sum_{\substack{r,s=k+1 \\ r<s}}^n c_{k,r} c_{k,s}\right) x_{n+k}^2\,.
  \end{equation}
  Further, from the identity (\ref{orientable}) we have,
  \begin{equation}\label{sub2}
   c_{n-1,n}=0.
  \end{equation}
  Now, by substituting (\ref{sqfree}) and (\ref{sub2}) in (\ref{fsw2}), it follows that $w_2(Y_n)$ can be identified with the class of the following term in $\mathcal{R/I}$ :
   \begin{equation}\label{fsw2'}
   \sum_{1\leq j<k\leq n-2} \left(\underbrace{\sum_{r=j+1}^ n\sum_{\substack {s= k+1 \\ s\neq r}}^n c_{j,r} c_{k,s}}_{I} ‎+\underbrace{c_{j,k}\cdot\sum_{\substack{r,s=k+1 \\ r< s}}^n c_{k,r} c_{k,s}}_{II}\right) x_{n+j}\,x_{n+k}\, .
   \end{equation}
   The expression $I$ in (\ref{fsw2'}) can be rewritten as follows :
   \begin{equation}\label{term1simp}
  \begin{split}
   \sum_{r=j+1}^n\sum_{\substack{s=k+1\\ s\neq r}}^n c_{j,r}\,c_{k,s} & = \sum_{r=j+1}^n c_{j,r}\left(\sum_{s=k+1}^n c_{k,s}-c_{k,r}\right)\\
   ~ & = \sum_{r=j+1}^n c_{j,r}\cdot\sum_{s=k+1}^n c_{k,s}-\sum_{r=j+1}^n c_{j,r}\,c_{k,r}\\
   ~ & = \sum_{r=1}^n c_{j,r}\,c_{k,r}=:P_{jk}
  \end{split}
  \end{equation}
  The last equality in (\ref{term1simp}) follows by condition (1) and the fact that $c_{i,j}=0$ for $i\geq j$.
  
  Again, using $c_{i,j}=0$ for $i\geq j$, the expression $II$ in (\ref{fsw2'}) can be rewritten as,
  \begin{equation}\label{term2simp}
   c_{j,k}\cdot\sum_{\substack{r,s=1 \\ r<s}}^n c_{k,r}\,c_{k,s}=:Q_{jk}.
  \end{equation}
  
  Further, it follows by the definition of $C$ and by (\ref{orientable}) that, 
  \begin{equation}\label{n-1nterms}
  P_{j\,n-1}=P_{jn}=Q_{j\,n-1}=Q_{jn}=0.
  \end{equation}
  
  Thus by (\ref{term1simp}), (\ref{term2simp}) and (\ref{n-1nterms}) it follows that $w_2(Y_n)$ can be identified with the class of the following term in $\mathcal{R/I}$:
  \begin{equation}
   \sum_{1\leq j<k\leq n} \left(P_{jk}+Q_{jk}\right) x_{n+j}\,x_{n+k}\,.
  \end{equation}
  Further, as a graded $\mathbb Z_2$-vector space $H^2(Y_n;\mathbb Z_2)$ is isomorphic to the subspace of $\mathcal{R/I}$ freely generated over $\mathbb Z_2$ by the classes of $x_{n+j}\,x_{n+k}\,$, $1\leq j<k\leq n$. Hence the theorem.
\end{proof}

\begin{remark}\label{spin23}
 The only oriented 2-dimensional real Bott manifold is the torus, which is classically known to be spin. The 3-dimensional oriented real Bott manifolds $Y_3$, correspond to the following two Bott matrices :
 \begin{equation*}
  \begin{pmatrix}
   0&0&0\\
   0&0&0\\
   0&0&0
  \end{pmatrix}\qquad
  \begin{pmatrix}
   0&1&1\\
   0&0&0\\
   0&0&0
  \end{pmatrix}
 \end{equation*}
 Here, we immediately see that $w_2(Y_3)=0$ and hence $Y_3$ admits a spin structure. This is a special case of the well known more general result of Steenrod that an oriented threefold is parallelizable.
\end{remark}

\begin{eg}\label{4stepeg}
 The 4-dimensional Bott manifolds admitting spin structure correspond to the following list of associated Bott matrices :
 \begin{equation*}
 \begin{pmatrix}
  0&0&0&0\\
  0&0&0&0\\
  0&0&0&0\\
  0&0&0&0
 \end{pmatrix}\quad
 \begin{pmatrix}
  0&0&0&0\\
  0&0&1&1\\
  0&0&0&0\\
  0&0&0&0
 \end{pmatrix}\quad
 \begin{pmatrix}
  0&0&1&1\\
  0&0&0&0\\
  0&0&0&0\\
  0&0&0&0
 \end{pmatrix}\quad
 \begin{pmatrix}
  0&0&1&1\\
  0&0&1&1\\
  0&0&0&0\\
  0&0&0&0
 \end{pmatrix}
 \end{equation*}
 \begin{equation*}
 \begin{pmatrix}
  0&1&1&0\\
  0&0&0&0\\
  0&0&0&0\\
  0&0&0&0
 \end{pmatrix}\quad
 \begin{pmatrix}
  0&1&1&0\\
  0&0&1&1\\
  0&0&0&0\\
  0&0&0&0
 \end{pmatrix}\quad
 \begin{pmatrix}
  0&1&0&1\\
  0&0&0&0\\
  0&0&0&0\\
  0&0&0&0
 \end{pmatrix}\quad
 \begin{pmatrix}
  0&1&0&1\\
  0&0&1&1\\
  0&0&0&0\\
  0&0&0&0
 \end{pmatrix}
 \end{equation*}
\end{eg}

\vspace{1mm}
\begin{remark}\label{4stepegrem}
Note that the above list of Bott matrices exhausts all orientable 4-dimensional real Bott manifolds. Thus it follows that every orientable 4-dimensional real Bott manifold is also spin. Moreover, it is known that a $4$-manifold $M$, is parallelizable if and only if it admits a spin structure (i.e $w_1(M)=w_2(M)=0$) and has vanishing Euler characteristic and signature ($\chi(M)=\sigma(M)=0$) (see \cite[Section 4]{hirzebruchhopf} and \cite[p. 699]{krupkahandbook}). Moreover, by Hirzebruch signature formula, $\sigma(M)=\frac{1}{3}p_1(M)[M]$, where $p_1(M)$ is the first Pontrjagin class and $[M]$ the fundamental class of $M$. Now, a real Bott manifold has vanishing Euler characteristic (since $\chi(Y_n)=\chi(Y_{n-1})\cdot\chi(S^1)$ and $\chi(S^1)=0$) and vanishing Pontrjagin classes by \cite[Corollary 6.8 (i)]{davisjanus}. Thus it follows that a $4$-dimensional real Bott manifold is orientable if and only if it is parallelizable. Further, it corresponds to one of the eight Bott matrices in the above list.
\end{remark}

The following example shows that this is not the case in dimensions $5$ and higher. Indeed there are $n$-dimensional Bott manifolds which are orientable but not spin when $n\geq 5$.

\begin{eg}\label{egorinotspinfornge5}
Let $Y_n$ be the $n$-dimensional Bott manifold, $n\geq 5$, associated to the Bott numbers $c_{1,2}=1$, $c_{1,n-2}=1$; $c_{n-2,n-1}=1$, $c_{n-2,n}=1$ and $c_{i,j}=0$ otherwise.  These numbers clearly satisfy (\ref{orientable}) but not (\ref{spin}).  Indeed in this case, when $j=1$ and $k=n-2$, the left hand side of (\ref{spin}) is $c_{1,n-2}\, c_{n-2,n-1}\,c_{n-2,n} \equiv1\bmod2.$
\end{eg}

\begin{definition}
We call the Bott matrix $C$ {\em spin} if and only if the associated Bott manifold $Y_n=Y_n(C)$ is spin.
\end{definition}

\begin{definition}
Let $R_i$ denote the $i$th row vector $(0,\ldots,0,0=c_{i,i},c_{i,i+1}, c_{i,i+2},\ldots, c_{i,n})$ of $C$. For every $1\leq j<k\leq n$, we define another $n\times n$ Bott matrix $C_{jk}$ with $R_j$ as the $j$th row and $R_k$ as the $k$th row and remaining rows with all entries $0$.
\end{definition}

\begin{corollary}\label{crspinsub} 
The Bott matrix $C$ is spin if and only if $C_{jk}$ is spin for every \\ $1\leq j<k\leq n$. 
\end{corollary}

\begin{proof}
From Theorem \ref{spinresult}, a necessary and sufficient condition for $C$ to be spin is that the entries $c_{i,j}, ~1\leq j\leq n$ on the row $R_i$ for every $1\leq i\leq n$ satisfy (\ref{orientable}) and further, the entries $c_{j,r}$, $1\leq r\leq n$ of $R_j$ and $c_{k,s}$, $1\leq s\leq n$ of $R_k$ for every $1\leq j<k\leq n$ satisfy (\ref{spin}).

Again by Theorem \ref{spinresult}, it follows that the necessary and sufficient condition for the Bott matrix $C_{jk}$ to be spin is that the entries $c_{j,r}$, $1\leq r\leq n$, of the $j$th row and the entries $c_{k,s}$, $1\leq s\leq n$, of the $k$th row of $C_{jk}$, satisfy (\ref{orientable}) and (\ref{spin}).  This can be readily seen because any row of $C_{jk}$, other than the $j$th or $k$th row, has all entries as $0$. Thus the entries on the $i$th row of $C_{jk}$ where $i\neq j,k$, trivially satisfy (\ref{spin}). Moreover, if either $i\neq j,k$ or $l\neq j,k$ and $1\leq i<l\leq n$, the entries of $C_{jk}$, on the $i$th and the $l$th row trivially satisfy (\ref{spin}). Hence the corollary.
\end{proof}

We state the following proposition without proof:
\begin{proposition}
 The $n$-dimensional Bott manifold can alternately be seen as the total space of a fibre bundle over $S^1$ with fibre an $(n-1)$-dimensional Bott manifold corresponding to the Bott matrix $C^{1}$ of size $(n-1)\times (n-1)$, defined by deleting the first row and first column of $C$. 
\end{proposition}

In this convention, we shall denote the $n$-dimensional Bott manifold by $Z_n$ and the fibre, which is the $(n-1)$-dimensional real Bott manifold associated to the matrix $C^{1}$, by $Z_{n-1}$. We can iterate this process and view $Z_{n-1}$ again as a fibre bundle over $S^1$ with fibre $Z_{n-2}$ which is the $(n-2)$-dimensional real Bott manifold associated to the matrix $C^{2}$ of size $(n-2)\times (n-2)$, obtained by deleting the first and the second rows and columns of $C$. Continuing this process we finally get that $Z_2$ is a two dimensional Bott manifold associated to the Bott matrix $C^{n-2}$, obtained by deleting the first $n-2$ rows and columns of $C$. Then $Z_2$ is a fibre bundle over $S^1$ with fibre $Z_1\simeq S^1$.

\begin{corollary}
 The $n$-dimensional real Bott manifold $Z_n$ is orientable (respectively spin) implies that the successive fibres $Z_{n-1}, Z_{n-2},\cdots, Z_{2}$ in the above iterated construction are all orientable (respectively spin).  
\end{corollary}
 
\begin{proof}
 This follows from (\ref{orientable}) and (\ref{spin}) since the Bott matrix corresponding to $Z_{k}$ is $C^{n-k}$ which is the matrix obtained from $C$ by deleting the first $k$ rows and $k$ columns. 
\end{proof}

\subsection{Spin condition for a more generally defined real Bott manifold}\label{genbottmatsection}

We recall here that Choi, Masuda and Oum give a more general definition of a Bott matrix in \cite{masudachoidigraph}. They call a square matrix $B$ to be a Bott matrix if there exists a permutation matrix $P$ and a strictly upper triangular binary matrix $C$ such that $B=PCP^{-1}$. They denote by $\mathcal B(n)$, the set of all such $n\times n$ matrices. Further, it follows from \cite[Section 3]{masudachoidigraph} that $B$ and $C$ are \emph{Bott equivalent}. 

Moreover, in \cite[Section 2]{masudachoidigraph} they also give a construction of a real Bott manifold $M(B)$ associated to $B$. In particular, when $B$ is a strictly upper triangular binary matrix then $M(B)$ is nothing but the associated real Bott manifold. 

In the following theorem we give a necessary and sufficient condition for the Bott manifold $M(B)$ to admit a spin structure where $B$ is any matrix in $\mathcal B(n)$.

\begin{theorem}\label{genspinresult}
 The real Bott manifold $M(B)$ associated to $B=(b_{i,j})\in\mathcal B(n)$ admits a spin structure if and only if the entries $b_{i,j}$ satisfy the following identities : 
 \begin{enumerate}
  \item For $1\leq i\leq n$,
  \begin{equation}\label{genorientable} \sum_{j=1}^n b_{i,j}\equiv0\bmod2. \end{equation}
  \item For $1\leq j<k\leq n$,
  \begin{equation}\label{genspin}
  \sum_{r=1}^n b_{j,r}\,b_{k,r}+\,b_{j,k}\cdot\sum_{\substack{r,s=1\\ r<s}}^n b_{k,r}\,b_{k,s}\equiv 0\bmod2.
  \end{equation} 
 \end{enumerate}
\end{theorem}

\begin{proof}
 Let $B\in\mathcal B(n)$. Let $P$ be an $n\times n$ permutation matrix such that $B=PCP^{-1}$ for an $n\times n$ strictly upper triangular binary matrix $C$.
 Let $C=(c_{i,j})$ and $\sigma\in S_n$ be the permutation corresponding to $P$. Note that by \cite[(3.1)]{masudachoidigraph},
                                                                                          \begin{equation}\label{sigma}
                                                                                           b_{\sigma(i),\sigma(j)}=c_{i,j}\quad\text{for}\quad 1\leq i,j\leq n
                                                                                          \end{equation}
  Now, by \cite[Theorem 1.6]{masudachoidigraph}, since $B$ and $C$ are Bott equivalent, $M(B)$ and $M(C)$ are affinely diffeomorphic. In particular, $M(C)$ is spin if and only if $M(B)$ is spin. The proof of the theorem now follows by substituting (\ref{sigma}) in (\ref{orientable}) and (\ref{spin}).
\end{proof}

We derive the following corollary, analogous to Corollary \ref{crspinsub}. We wish to remark here that this result has been proved in \cite[Theorem 1.2]{gasiorspin} using different techniques. We omit the proof which is similary to that of Corollary \ref{crspinsub}. 

\begin{corollary}\label{gencrspinsub}
 The Bott matrix $B$ is spin if and only if $B_{jk}$ is spin for every \\ $1\leq j<k\leq n$, where $B_{jk}$ is the $n\times n$ matrix having same $j$th and $k$th row as $B$ and all other rows zero.
\end{corollary}

\begin{remark}\label{BE45}
Recall that, by \cite[Table 1]{masudachoidigraph}, \cite[Section 7]{masudakamishimacohorigid}, \cite[Section 3]{nazradiffeo}, any $B\in\mathcal B(4)$ with $M(B)$ orientable, is Bott equivalent to one of the following three Bott matrices :

  \begin{equation}\label{BE4}
   \begin{pmatrix}
    0&0&0&0\\
    0&0&0&0\\
    0&0&0&0\\
    0&0&0&0
   \end{pmatrix}\quad
   \begin{pmatrix}
    0&1&0&1\\
    0&0&1&1\\
    0&0&0&0\\
    0&0&0&0
   \end{pmatrix}\quad
   \begin{pmatrix}
    0&0&0&0\\
    0&0&1&1\\
    0&0&0&0\\
    0&0&0&0 
   \end{pmatrix}
  \end{equation}
  Also, by \cite[Section 3]{nazradiffeo} and \cite[Table 1]{masudachoidigraph}, any $B\in\mathcal B(5)$ with $M(B)$ orientable, is Bott equivalent to one of the following eight Bott matrices :
  \begin{equation}\label{BE5}
  \begin{split}
   \begin{pmatrix}
    0 & 0 & 0 & 0 & 0\\
    0 & 0 & 0 & 0 & 0\\
    0 & 0 & 0 & 0 & 0\\
    0 & 0 & 0 & 0 & 0\\
    0 & 0 & 0 & 0 & 0
   \end{pmatrix}\quad
   \begin{pmatrix}
    0 & 0 & 0 & 0 & 0\\
    0 & 0 & 1 & 1 & 0\\
    0 & 0 & 0 & 1 & 1\\
    0 & 0 & 0 & 0 & 0\\
    0 & 0 & 0 & 0 & 0
   \end{pmatrix}\quad
   \begin{pmatrix}
    0 & 1 & 1 & 1 & 1\\
    0 & 0 & 0 & 0 & 0\\
    0 & 0 & 0 & 0 & 0\\
    0 & 0 & 0 & 0 & 0\\
    0 & 0 & 0 & 0 & 0
   \end{pmatrix}\quad
   \begin{pmatrix}
    0 & 0 & 0 & 0 & 0\\
    0 & 0 & 0 & 0 & 0\\
    0 & 0 & 0 & 1 & 1\\
    0 & 0 & 0 & 0 & 0\\
    0 & 0 & 0 & 0 & 0
   \end{pmatrix}\\
   \begin{pmatrix}
    0 & 1 & 1 & 0 & 0\\
    0 & 0 & 1 & 0 & 1\\
    0 & 0 & 0 & 1 & 1\\
    0 & 0 & 0 & 0 & 0\\
    0 & 0 & 0 & 0 & 0
   \end{pmatrix}\quad
   \begin{pmatrix}
    0 & 1 & 1 & 0 & 0\\
    0 & 0 & 0 & 0 & 0\\
    0 & 0 & 0 & 1 & 1\\
    0 & 0 & 0 & 0 & 0\\
    0 & 0 & 0 & 0 & 0
   \end{pmatrix}\quad
   \begin{pmatrix}
    0 & 0 & 1 & 0 & 1\\
    0 & 0 & 1 & 1 & 0\\
    0 & 0 & 0 & 1 & 1\\
    0 & 0 & 0 & 0 & 0\\
    0 & 0 & 0 & 0 & 0
   \end{pmatrix}\quad
   \begin{pmatrix}
    0 & 0 & 1 & 0 & 1\\
    0 & 0 & 1 & 1 & 0\\
    0 & 0 & 0 & 0 & 0\\
    0 & 0 & 0 & 0 & 0\\
    0 & 0 & 0 & 0 & 0
   \end{pmatrix}
   \end{split}
   \end{equation}
  We can readily check that all the three matrices in (\ref{BE4}) are spin and only the first four matrices in (\ref{BE5}) are spin.
\end{remark}

Following \cite[Table 1]{masudachoidigraph}, we let $\mathcal Spin_n$ denote the number of $n$-dimensional real spin Bott manifolds up to diffeomorphism. Then, by \cite[Example 3.1]{masudachoidigraph}, Remark \ref{spin23} and Remark \ref{BE45} above, we have the following table :
\vspace{3mm}
\begin{center}
\begin{tabular}{c|l l l l l }
\hline
$n$ & 1 & 2 & 3 & 4 & 5 \\ \hline
$\mathcal O_n$ & 1 & 1 & 2 & 3 & 8 \\ \hline
$\mathcal{S}pin_n$ & 1 & 1 & 2 & 3 & 4 \\ \hline
 
\end{tabular}
\end{center}

\section{Digraph characterization of Spin structure on real Bott manifolds}
 We begin this section by recalling the definition of an acyclic digraph associated to a real Bott manifold. The dictionary between a real Bott manifold and  its associated acyclic digraph, via the Bott matrix, established in the work of Choi, Masuda and Oum \cite[Section 4]{masudachoidigraph}, gives a new way of studying the topology of these manifolds by means of the combinatorics of the associated digraph. The main theorem (Theorem \ref{digraphspin}) in this section gives a combinatorial criterion on the acyclic digraph, which characterizes the existence of spin structure on the associated real Bott manifold. 

\begin{definition}
 A directed graph (digraph) is a tuple $(V,E)$ consisting of a set $V$, of elements, called vertices and a set $E$, of ordered pairs of distinct vertices, called edges. 
\end{definition}

\begin{definition} \label{graphtheorydefinitions} Let $D=(V,E)$ be a digraph with {\it vertices} $V=\{u_1,\cdots,u_n\}$ and edges $(u_i,u_j)$ indexed by an ordered pair of distinct vertices. In particular, we assume that $D$ has no loops and has at most one directed edge between any pair of vertices. The {\it adjacency matrix} $A(D)$ associated to $D$ is therefore an $n\times n$ matrix in $M_n(\mathbb{Z}_2)$ with diagonal entries zero. Conversely, to any such matrix $A=(a_{i,j})\in M_n(\mathbb{Z}_2)$ we associate a graph with $n$ vertices and an edge from $u_i$ to $u_j$ if and only if $a_{i,j}=1$, $1\leq i,j\leq n$. In particular, given a Bott matrix $C=(c_{i,j})$ (see \ref{Bottmatrix}) we can associate a digraph, $D_C$ to it, having $n$ vertices. Moreover, since the matrix is strictly upper triangular, the digraph $D_C$ admits an ordering of vertices $u_1,\cdots,u_n$ such that $i<j$ whenever there is an edge from $u_i$ to $u_j$. In particular, $D_C$ is an \emph{acyclic digraph} (see \cite[Section 4]{masudachoidigraph}).  

\end{definition}

\subsection{Notations :} Let $D_C$ be an acyclic digraph associated to a Bott matrix $C$. For each vertex $u_i$ of $D_C$ we denote by,
\begin{align*}
 N^+_{D_C}(u_i):=\{\, u_j\, |\, c_{i,j}=1\,\} \quad \text{and} \quad N_i:=|N^+_{D_C}(u_i)|\text{ is the out degree of }u_i\\
 N^-_{D_C}(u_i):=\{\, u_j\, |\, c_{j,i}=1\,\} \quad \text{and} \quad I_i:=|N^-_{D_C}(u_i)|\text{ is the in degree of }u_i
\end{align*}  (see \cite[Section 4]{masudachoidigraph})

For each pair of vertices $u_i,u_j$ we further denote by $M_{ij}:=|N^+_{D_C}(u_i)\cap N^+_{D_C}(u_j)|$. More precisely, $M_{ij}$ is the number of vertices $u_k$ which are the out neighbours of both $u_i$ and $u_j$.

The following two lemmas respectively reinterpret the terms $P_{jk}$ and $Q_{jk}$ on the left hand side of the identity (\ref{spin}), in terms of the combinatorial data of the digraph.

\begin{lemma} \label{Lemma1}
  \begin{equation} \label{term1} P_{jk}=M_{jk} \end{equation}
\end{lemma}

\begin{proof} 
Note that the product $c_{k,r}\,c_{j,r}\neq0$ if and only if $c_{k,r}=1=c_{j,r}$, that is if and only if there is an edge from $u_j$ to $u_r$ as well as $u_k$ to $u_r$. Thus the sum $\sum\limits_{r=1}^nc_{j,r}\,c_{k,r}=P_{jk}$ counts the number of vertices $\{u_r\}$ that have edges from $u_j$ as well as $u_k$ coming into it. This number is precisely $M_{jk}$. Hence the lemma.
\end{proof}

\begin{lemma} \label{Lemma2}
  \begin{equation} \label{term2} Q_{jk}=c_{j,k}\cdot{N_k\choose2} \end{equation}
\end{lemma}
\begin{proof} Note that the number of unordered pairs of distinct edges coming out of $u_k$ in $D_C$ is precisely ${N_k\choose 2}$. Also, the product $c_{k,r}\,c_{k,s}\neq 0$ if and only if $c_{k,r}=1=c_{k,s}$. Thus the sum $\sum\limits_{\substack{r,s=1\\ r<s}}^n c_{k,s}\, c_{k,r}=Q_{jk}$ counts the total number of unordered pairs of distinct edges coming out of $u_k$. Hence the lemma.
\end{proof}

The next theorem reformulates Theorem \ref{spinresult} in terms of the associated digraph $D_C$.

\begin{theorem} \label{digraphspin}
 The $n$-dimensional Bott manifold $Y_n(C)$ admits a spin structure if and only if for the corresponding digraph $D_C$ the following two conditions are true :
 \begin{enumerate}
  \item $N_k$ is even for all $1\leq k\leq n$.
  \item $M_{jk}$ and $c_{j,k}\cdot{N_k\choose2}$ have the same parity for all $1\leq j<k\leq n$.
 \end{enumerate}
\end{theorem}

\begin{proof}
  Note that, the out degree $N_i$ of $u_i$ is $\sum_{i=1}^n c_{i,j}$. Thus the identity (\ref{orientable}) in the statement of Theorem \ref{spinresult} is equivalent to condition (1) above. Furthermore, it follows by Lemma \ref{Lemma1} and Lemma \ref{Lemma2} that the identity (\ref{spin}) in the statement of Theorem \ref{spinresult} is equivalent to condition (2) above. Hence the theorem.
\end{proof}

\begin{remark}
 Condition (2) in Theorem \ref{digraphspin} can be made more explicit as follows :
 \begin{enumerate}
  \item When $N_k=4m$ then ${N_k\choose2}$ is always even. So condition (2) is equivalent to saying that $M_{jk}$ is even.
  \item When $N_k=4m-2$ the ${N_k\choose2}$ is always odd. So condition (2) is equivalent to saying that $M_{jk}$ is even when there is no edge from $u_j$ to $u_k$ in $D_C$ and $M_{jk}$ is odd when there is an edge from $u_j$ to $u_k$ in $D_C$.
 \end{enumerate}
\end{remark}

We will now look at some examples to illustrate Theorem \ref{digraphspin}. 

\begin{eg}\label{egdigraph} \hfill
\begin{enumerate}
\item \begin{tabular}{p{6cm}c}
           $C_1= \begin{pmatrix}
            0 & 0 & 0 & 0 & 0 & 0 \\
            0 & 0 & 0 & 0 & 1 & 1 \\
            0 & 0 & 0 & 0 & 1 & 1 \\
            0 & 0 & 0 & 0 & 1 & 1 \\
            0 & 0 & 0 & 0 & 0 & 0 \\
            0 & 0 & 0 & 0 & 0 & 0
           \end{pmatrix} $
           &
           $D_{C_1} :\quad\vcenter{\hbox{\begin{tikzpicture}[roundnode/.style={circle, fill=black!70, inner sep=0pt, minimum size=4mm}]
           \foreach \i [count=\ni] in {120, 60, ..., -180}
           \node[circle,fill,inner sep=2pt, label=\i:{$u_\ni$}] at (\i:2cm) (u\ni) {};

            \draw[->] (u2) -- (u5);
            \draw[->] (u2) -- (u6);
            \draw[->] (u3) -- (u5);
            \draw[->] (u3) -- (u6);
            \draw[->] (u4) -- (u5);
            \draw[->] (u4) -- (u6); 
          \end{tikzpicture} }} $   
    \end{tabular}\vspace{3mm}
\\
    Here we have $N_{D_{C_1}}^+(u_i)=\{u_5,u_6\}$ for $i=2,3,4$ and $N_{D_{C_1}}^+(u_i)=\emptyset$ for $i=1,5,6$. \\
    Clearly, $N_i=|N_{D_{C_1}}^+(u_i)|$ is even for all $1\leq i\leq 5$.\\
    Also, $M_{jk}=|N_{D_{C_1}}^+(u_j)\cap N_{D_{C_1}}^+(u_k)|=2$ for $2\leq j<k\leq4$ and $M_{jk}=0$ , otherwise.\\
    When $j=2$ and $k=3$ we get that $M_{23}=2$ and $c_{2,3}\cdot{N_3\choose2}=0$ have the same parity.\\
    When $j=2$ and $k=4$ we get that $M_{24}=2$ and $c_{2,4}\cdot{N_4\choose2}=0$ have the same parity.\\
    When $j=3$ and $k=4$ we get that $M_{34}=2$ and $c_{3,4}\cdot{N_4\choose2}=0$ have the same parity.\\
    For other pairs $j<k$ we get $M_{jk}=0$ and $c_{j,k}\cdot{N_k\choose2}=0$ have the same parity.\\
    Thus the associated Bott manifold admits a spin structure.\\

\item \begin{tabular}{p{6cm}c}
       $ C_2= \begin{pmatrix}
            0 & 1 & 1 & 1 & 1 & 0 \\
            0 & 0 & 1 & 0 & 0 & 1 \\
            0 & 0 & 0 & 0 & 1 & 1 \\
            0 & 0 & 0 & 0 & 0 & 0 \\
            0 & 0 & 0 & 0 & 0 & 0 \\
            0 & 0 & 0 & 0 & 0 & 0 \\
           \end{pmatrix} $
           &$D_{C_2} : \quad\vcenter
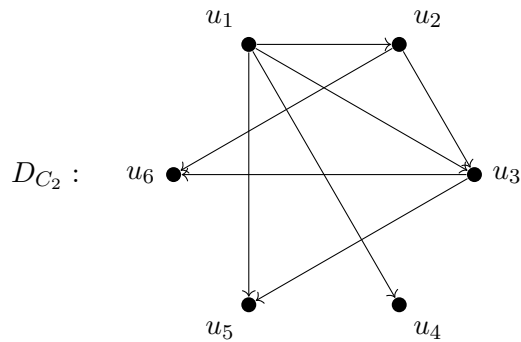
{\hbox{\begin{tikzpicture}[roundnode/.style={circle, fill=black!70, inner sep=0pt, minimum size=4mm}]
	     \foreach \i [count=\ni] in {120, 60, ..., -180}
	     \node[circle,fill,inner sep=2pt, label=\i:{$u_\ni$}] at (\i:2cm) (u\ni) {};
 
			    \draw[->] (u1) -- (u2);
    			    \draw[->] (u1) -- (u5);
			    \draw[->] (u1) -- (u3);
			    \draw[->] (u1) -- (u4);
			    \draw[->] (u2) -- (u3);          
			    \draw[->] (u2) -- (u6);
			    \draw[->] (u3) -- (u5); 
			    \draw[->] (u3) -- (u6);
			  \end{tikzpicture} }} $  
      \end{tabular}\vspace{2mm}
\\
    Here we have $N_{D_{C_2}}^+(u_1)=\{u_2,u_3,u_4,u_5\}$ , $N_{D_{C_2}}^+(u_2)=\{u_3,u_6\}$ , $N_{D_{C_2}}^+(u_3)=\{u_5,u_6\}$ and $N_{D_{C_2}}^+(u_i)=\emptyset$ for $i=4,5,6$. \\
    Clearly, $N_i=|N_{D_{C_2}}^+(u_i)|$ is even for all $1\leq i\leq 6$.\\
    When $j=1$ and $k=2$ we have $M_{12}=|N_{D_{C_2}}^+(u_1)\cap N_{D_{C_2}}^+(u_2)|=|\{u_3\}|=1$ and $c_{1,2}\cdot{N_2\choose2}=1$ have the same parity.\\
    When $j=1$ and $k=3$ we have $M_{13}=|N_{D_{C_2}}^+(u_1)\cap N_{D_{C_2}}^+(u_3)|=|\{u_5\}|=1$ and $c_{1,3}\cdot{N_3\choose2}=1$ have the same parity.\\
    When $j=2$ and $k=3$ we have $M_{23}=|N_{D_{C_2}}^+(u_2)\cap N_{D_{C_2}}^+(u_3)|=|\{u_6\}|=1$ and $c_{2,3}\cdot{N_3\choose2}=1$ have the same parity.\\
    For other pairs $j<k$ we have $M_{jk}=0$ and $c_{j,k}\cdot{N_k\choose2}=0$ have the same parity.\\
    Thus the corresponding Bott manifold admits a spin structure.\\
    
\item \begin{tabular}{p{6cm}c}
       $C_3= \begin{pmatrix}
            0 & 0 & 1 & 0 & 1 \\
            0 & 0 & 1 & 1 & 0 \\
            0 & 0 & 0 & 1 & 1 \\
            0 & 0 & 0 & 0 & 0 \\
            0 & 0 & 0 & 0 & 0
           \end{pmatrix} $
           &
           $D_{C_3} :\quad\vcenter
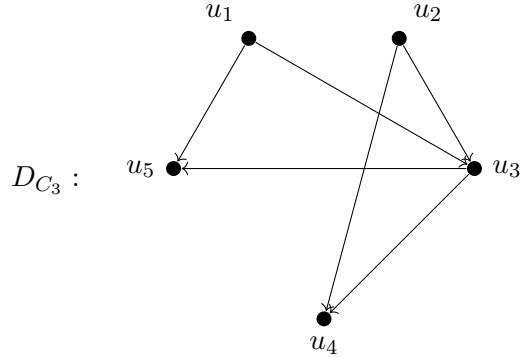
{\hbox{\begin{tikzpicture}[roundnode/.style={circle, fill=black!70, inner sep=0pt, minimum size=4mm}]
	    \foreach \i [count=\ni] in {120, 60, 0, -90 ,-180}
	    \node[circle,fill,inner sep=2pt, label=\i:{$u_\ni$}] at (\i:2cm) (u\ni) {};
 
                            \draw[->] (u1) -- (u3);
                            \draw[->] (u1) -- (u5);
                            \draw[->] (u2) -- (u3);
                            \draw[->] (u2) -- (u4);
                            \draw[->] (u3) -- (u4);
                            \draw[->] (u3) -- (u5);
                           \end{tikzpicture}}}$
           \end{tabular}\vspace{2mm}
\\
         Here we have $N_{D_{C_3}}^+(u_1)=\{u_3,u_5\}$ , $N_{D_{C_3}}^+(u_2)=\{u_3,u_4\}$ , $N_{D_{C_3}}^+(u_3)=\{u_4,u_5\}$ and $N_{D_{C_3}}^+(u_i)=\emptyset$ for $i=4,5$.\\
         Clearly, $N_i=|N_{D_{C_3}}^+(u_i)|$ is even for all $1\leq i\leq 4$.\\
         When $j=1$ and $k=2$ we get $M_{12}=|N_{D_{C_3}}^+(u_1)\cap N_{D_{C_3}}^+(u_2)|=\{u_3\}|=1$ and $c_{1,2}\cdot{N_2\choose2}=0$ do not have the same parity.\\
         Thus the associated Bott manifold does not admit a spin structure.\\
         
\item \begin{tabular}{p{6cm}c}
        $C_4= \begin{pmatrix}
            0 & 0 & 0 & 0 & 0 & 0 & 0 \\
            0 & 0 & 1 & 1 & 1 & 1 & 0\\
            0 & 0 & 0 & 1 & 1 & 1 & 1\\
            0 & 0 & 0 & 0 & 0 & 0 & 0\\
            0 & 0 & 0 & 0 & 0 & 1 & 1\\
            0 & 0 & 0 & 0 & 0 & 0 & 0\\
            0 & 0 & 0 & 0 & 0 & 0 & 0
           \end{pmatrix}$ 
           &
         $D_{C_4} :\quad\vcenter
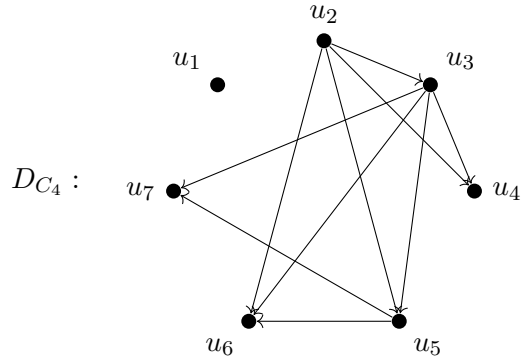
{\hbox{\begin{tikzpicture}[roundnode/.style={circle, fill=black!70, inner sep=0pt, minimum size=4mm}]
	  \foreach \i [count=\ni] in {135, 90, 45, 0, -60, -120 ,-180}
	  \node[circle,fill,inner sep=2pt, label=\i:{$u_\ni$}] at (\i:2cm) (u\ni) {};
 
			  \draw[->] (u2) -- (u3);
			  \draw[->] (u2) -- (u4);
			  \draw[->] (u2) -- (u5);
			  \draw[->] (u2) -- (u6);
			  \draw[->] (u3) -- (u4);
			  \draw[->] (u3) -- (u5);
			  \draw[->] (u3) -- (u6);
			  \draw[->] (u3) -- (u7);
			  \draw[->] (u5) -- (u6);
			  \draw[->] (u5) -- (u7);
                         \end{tikzpicture}}}$

         \end{tabular}\vspace{3mm}
         \\
         Here we have $N_{D_{C_4}}^+(u_2)=\{u_3,u_4,u_5,u_6\}$ , $N_{D_{C_4}}^+(u_3)=\{u_4,u_5,u_6,u_7\}$ , $N_{D_{C_4}}^+(u_5)=\{u_6,u_7\}$ and $N_{D_{C_4}}^+(u_i)=\emptyset$ for $i=1,4,6,7$\\
         Clearly, $N_i=|N_{D_{C_4}}^+(u_i)|$ is even for all $1\leq i\leq 6$\\
         When $j=2$ and $k=3$ we get $M_{23}=|N_{D_{C_4}}^+(u_2)\cap N_{D_{C_4}}^+(u_3)|=|\{u_3,u_5,u_6\}|=3$ and $c_{2,3}\cdot{N_3\choose2}=6$ do not have the same parity.\\
         Thus the associated Bott manifold does not admit a spin structure.
\end{enumerate}
\end{eg}

\section{Higher Stiefel-Whitney classes and Stiefel-Whitney numbers}
 We have the following result for the $(n-1)^\text{th}$ Stiefel-Whitney class of $Y_n$:
 
\begin{theorem}\label{n-1th sfclass} 
\begin{enumerate}
 \item[(i)] We have the following formula for $w_{n-1}(Y_n)$ in $H^{*}(Y_n;\mathbb Z_2)$ in terms of the Bott numbers $c_{i,j}$: \[ w_{n-1}(Y_n)=c_{1,2}\cdot c_{2,3}\cdots c_{n-1,n} \cdot x_{n+1}\cdot x_{n+2}\cdots x_{2n-1}.\]
 \item[(ii)] If $Y_n$ is an oriented real Rott manifold then $w_{n-1}(Y_n)=0$.
 \item[(iii)] We have $w_{n-1}(Y_n)=0$ if and only if there exists a pair of vertices $u_i,u_{i+1}$ in the digraph $D_C$ with no edge from $u_i$ to $u_{i+1}$.
 \end{enumerate}
\end{theorem}

\begin{proof} 
 The proof of (i) follows by induction on $n$ using (\ref{rf2}) and the fact that in $H^*(Y_n;\mathbb Z_2)$ the following relations hold : \begin{equation}\label{sqrelns} 
 x_{n+1}^2=0 ;\, x_{n+1}\cdot x_{n+2}^2=0;\, x_{n+1}\cdot x_{n+2}\cdot x_{n+3}^2=0;\, \cdots ;\,x_{n+1}\cdot x_{n+2}\cdots x^2_{2n-2}=0.                                                                                                                                      \end{equation}
 From (\ref{orientable}), $c_{n-1,n}=0$ if $Y_n$ is orientable. Hence (ii) follows from (i). Also (iii) follows from (i) and the definition of $D_C$.
\end{proof}

\begin{remark}
 The assertion (ii) of Theorem \ref{n-1th sfclass} is true for any even dimensional manifold but not in general true when the dimension is odd (see \cite[Theorem II and examples on p. 94]{masseystiefel}). 
\end{remark}

\begin{remark}
 We hope to compute closed formulae for $w_k(Y_n)$, for $k\geq 3$ and also characterize their vanishing in terms of the corresponding digraph. For instance when $Y_n$ is orientable, by Wu's formula \cite[p. 96]{milnorcclasses}, $w_3(Y_n)=\operatorname{Sq}^1(w_2(Y_n))$. In particular, if $w_1(Y_n)=w_2(Y_n)=0$ then $w_3(Y_3)=0$.
\end{remark}

\subsection{Real Bott manifolds bound}
\begin{definition}\label{ncBT} 
 A compact $n$-dimensional manifold $M$ without boundary is null-cobordant if it is diffeomorphic to the boundary of some compact smooth $(n+1)$-dimensional manifold $\mathcal{W}$ with boundary.  
\end{definition}

Let $w_{k}:=w_{k}(Y_n)$ for $1\leq k\leq n$. Also let $\mu_{Y_n}$ denote the fundamental class of $Y_n$ in $H_n(Y_n;\mathbb Z_2)$. Then
\begin{equation}\label{swnos} \langle w_1^{r_1} \cdots w_n^{r_n},\mu_{Y_n}\rangle \in \mathbb Z_2 \end{equation}
such that $\sum_{i=1}^n i\cdot r_i=n$ are the Steifel-Whitney numbers of $Y_n$.

\begin{theorem}\label{vanishswnosBT}
 Any real Bott manifold is null-cobordant.  
\end{theorem}

\begin{proof}
 From the second part of Theorem \ref{simrf}, it follows that, any monomial $w_1^{r_1}\cdots w_n^{r_n}$ of total dimension $n$ in $\mathcal {R/I}$ is a $\mathbb Z_2$-linear combination of square free monomials of degree $n$ in $x_{n+1},x_{n+2},\ldots, x_{2n-1}$. But there are no square free monomials of degree $n$ in $x_{n+j}, ~1\leq j\leq n-1$.  Thus the monomial $w_1^{r_1}\cdots w_n^{r_n}=0$ in $H^n(Y_n;\mathbb Z_2)$ so that the associated Stiefel-Whitney number is zero. Therefore by Thom's theorem it follows that $Y_n$ is null-cobordant.
\end{proof}

\begin{definition}\label{oncBT}
 A compact oriented $n$-dimensional manifold $M'$ without boundary is orientedly null-cobordant if it is diffeomorphic to the boundary of some compact smooth $(n+1)$-dimensional oriented manifold $\mathcal{W'}$ with boundary.
\end{definition}

Let $Y_{n}$ denote an oriented $n$-dimensional real Bott manifold.  Let $p_{i}:=p_i(Y_{n})$ denote the $i$th Pontrjagin class of $Y_{n}$ in $H^{4i}(Y_{n},\mathbb Z)$ and $\mu_{Y_{n}}$ denote the fundamental homology class in $H_{n}(Y_{n},\mathbb Z)$. Then for each $I=i_1,\ldots,i_r$ a partition of $k$, the $I$th Pontrjagin number of $Y_{n}$ is given by
\begin{equation}\label{pnos} \langle p_{i_1}\cdots p_{i_r},\mu_{Y_{n}}\rangle\in \mathbb Z\end{equation}
when $n=4k$. It is zero when $n$ is not divisible by $4$. Here $v[Y_n]:=\langle v,\mu_{Y_n}\rangle$ denotes the Kronecker index of any $v\in H^n(Y_n;\mathbb Z)$ (see \cite[p. 185]{milnorcclasses}).

\begin{corollary}\label{orcoborBT}
 Any oriented real Bott manifold is orientedly null-cobordant.
\end{corollary}

\begin{proof}
 Note that \cite[Corollary 6.8(i)]{davisjanus}, implies that all the Pontrjagin numbers of $Y_{n}$ vanish. Moreover, we have shown above in the proof of Theorem \ref{vanishswnosBT} that all the Stiefel-Whitney numbers of $Y_{n}$ vanish. Thus the corollary follows by Wall's theorem \cite[Section 8, Corollary 1]{wallcobordism}.
\end{proof}

\begin{remark}
 More generally, it follows from \cite[Corollary 6.8(i)]{davisjanus} and Wall's theorem that if an orientable small cover is null cobordant then it is orientedly null cobordant.
\end{remark}

\begin{remark}
 There are examples of small covers whose top Stiefel-Whitney class does not vanish. For example the non-orientable surfaces of odd genus ( see \cite[Example 1.20]{davisjanus}) have non-vanishing second Stiefel-Whitney class.
\end{remark}

\begin{remark}
 We wish to mention here that Theorem \ref{vanishswnosBT} was proved jointly with V Uma in \cite[Theorem 4.24]{umaraisa}. Indeed a more recent preprint has appeared where the same result is proved in \cite[Theorem 2.5]{lucobordism}. Also using similar techniques L\"{u} extends this result to small covers over $P^n\times\Delta^1$ where $P^n$ is a product of simplices (see \cite[Theorem 3.4]{lucobordism}). From \cite{lucobordism} we came to know that Cheng and Wang \cite{chenwang} and L\"{u} and Tan \cite{lutan} have proved that any real Bott manifold bounds \emph{equivariantly} using different methods. However our Corollary \ref{orcoborBT} regarding oriented cobordism is new and could not be explicitly found in the papers \cite{chenwang, lutan, lucobordism}
\end{remark}

\noindent
{\bf Acknowledgement:} I am grateful to my advisor V. Uma for her valuable guidance. I also thank Prof. P. Sankaran for valuable discussions, a careful reading of the manuscript and his comments and suggestions and Prof. M. Masuda for his valuable comments and suggestions on earlier versions of the manuscript. I finally thank University Grants Commission (UGC), India for financial assistance.

\bibliographystyle{siam}
\bibliography{refs_raisa}

\end{document}